\documentclass[11pt,a4paper]{article}

\usepackage{epsf,epsfig,amsfonts,amsgen,amsmath,amstext,amsbsy,amsopn,amsthm}
\usepackage{amsmath,times,mathptmx}
\usepackage{amsfonts,amsthm,amssymb}
\usepackage{graphicx}
\usepackage{latexsym,bm}
\usepackage{indentfirst}
\usepackage{color}
\usepackage[colorlinks=true,anchorcolor=blue,filecolor=blue,linkcolor=blue,urlcolor=blue,citecolor=blue]{hyperref}
\usepackage{float}
\usepackage{tikz}
\usepackage{verbatim}
\usepackage{mathrsfs}
\usepackage[nameinlink,noabbrev]{cleveref}
\usepackage{enumitem}

\newcommand{\floor}[1]

\evensidemargin=\oddsidemargin
\newtheorem{theorem}{Theorem}[section]
\newtheorem{problem}{Problem}[section]

\newtheorem{lemma}{Lemma}[section]
\newtheorem{corollary}{Corollary}[section]

\addtocounter{section}{0}

\begin{document}

\title{Nikiforov's spectral consecutive cycle problem and the connected-matching method}

\author{{\bf Bo Ning}\thanks{College of Computer Science, Nankai University, Tianjin 300350, PR China. Supported by the National 
NSFC (No. 12371350). Email: bo.ning@nankai.edu.cn (B. Ning).}~~~~ {\bf Mingqing Zhai}\thanks{School of Mathematics and Statistics, Nanjing University of Science and Technology. Nanjing 210094, Jiangsu, PR China. Corresponding author. Supported by NSFC (No. 12571369). E-mail address: mqzhai@njust.edu.cn (M. Zhai).}}
\date{}

\maketitle

\begin{abstract}
Let $\rho(G)$ denote the adjacency spectral radius of a graph $G$ of order $n$.
We determine the sharp constant in an open problem of Nikiforov (2008) on cycles of
consecutive lengths.  
For every $\varepsilon>0$ and all sufficiently large $n$, 
if $G$ is an $n$-vertex graph with $\rho(G)>\sqrt{\lfloor{n^2/4}\rfloor},$
then $G$ contains a cycle $C_g$ for every integer length
$3\le g\le (\frac{3-\sqrt5}{2}-\varepsilon)n.$
The constant $(3-\sqrt5)/2$ is the best possible, as shown by the split graph
$K_k\vee\overline K_{n-k}$ with $k\sim(3-\sqrt5)n/4$. 
Our result improves all previous results [LAA2008, CPC2020, JGT2023, JGT2023, GC2024].
The proof combines the degree form of Szemer\'edi's regularity lemma, 
a spectral matching theorem of Feng-Yu-Zhang, Weyl’s inequality, 
a refinement of \L{}uczak's connected matching embedding method, and other ideas.
\end{abstract}

\begin{flushleft}

\end{flushleft}

\maketitle

\section{Introduction}\label{sec:introduction}

Unless otherwise explicitly stated\footnote{In this paper, in the statement and proof of Lemma \ref{lem:closed-walk}, 
we need the concept of multi-graphs}, all graphs considered in this paper are finite and simple.
For a graph $G$, let $V(G)$
and $E(G)$ denote its vertex set and edge set, respectively, and write
$n(G)=|V(G)|$ and $e(G)=|E(G)|$.  
The minimum degree, matching number, and
adjacency spectral radius of $G$ are denoted by $\delta(G)$, $\nu(G)$, and
$\rho(G)$, respectively.  
For a path or a walk $P$, let $\ell(P)=|E(P)|$.  
As usual, $C_g$ denotes a cycle of length $g$, and
$G\vee H$ denotes the join of two vertex-disjoint graphs $G$ and $H$.  
A block in $G$ is a maximal $2$-connected subgraph of $G$.

The distribution of cycle lengths in dense graphs is a classical topic in extremal graph theory. 
A graph $G$ is called \emph{pancylic}, if it contains all cycles $C_{g}$ for each length $g\in [3,n]$. 
This concept was introduced by Bondy in 1971  \cite{Bondy1971}. 
In the same paper \cite{Bondy1971}, 
Bondy proved that every $n$-vertex Hamiltonian graph $G$ is pancyclic, if $e(G)\geq \frac{n^2}{4}$, 
unless $G=K_{\frac{n}{2},\frac{n}{2}}$ where $n$ is even. 
As a consequence, the following result is true: 
every $n$-vertex Hamiltonian graph $G$ is pancyclic if every non-adjacent vertices have degree sum at least $n$, 
unless $G=K_{\frac{n}{2},\frac{n}{2}}$ where $n$ is even. When we remove the condition of ``$G$ is Hamiltonian", 
a classical result in Bollob\'as's textbook \cite{Bollobas1978} showed that 
an $n$-vertex graph with $e(G)>\frac{n^2}{4}$ contains all cycle $C_{g}$ for $g\in [3,\lfloor\frac{n+3}{2}\rfloor]$. 
Nikiforov \cite{Nikiforov2008} proposed
the corresponding spectral problem in 2008, 
which is listed as Problem 3 in \cite{LiuNing2023}: 

\begin{problem}[Nikiforov \cite{Nikiforov2008}]\label{prob:main}
Determine the largest constant $C$ such that, for every $\varepsilon>0$ and all sufficiently
large $n$, any graph $G$ on $n$ vertices whose spectral radius satisfies
\begin{equation}\label{eq:threshold}
 \rho(G)>\sqrt{\lfloor{n^2/4}\rfloor}
\end{equation}
contains a cycle $C_g$ for every integer length $3\le g\le(C-\varepsilon)n$.
\end{problem}

The condition stated in Problem \ref{prob:main} is sharp: 
consider the complete bipartite graph $K_{\lceil\frac{n}{2}\rceil,\lfloor\frac{n}{2}\rfloor}$.
For this example, we have $\rho(K_{\lceil\frac{n}{2}\rceil,\lfloor\frac{n}{2}\rfloor})=\sqrt{\lfloor n^2/4\rfloor}$, 
and it contains no odd cycles.

By means of deleting edges with the least eigencomponent, Nikiforov~\cite{Nikiforov2008} established the first lower bound $C\ge 1/320$. 
This bound was subsequently improved to $1/160$ by Ning and Peng~\cite{NingPeng2020}, 
who employed ideas from generalized Tur\'an-type theorems together with a theorem recorded in Bollob\'as \cite{Bollobas1978}. 
Employing the classical eigenvector method, Zhai and Lin~\cite{ZhaiLin2023} further raised the bound to $1/7$. 
Independently and simultaneously, Li and Ning~\cite{LiNing2023} improved Nikiforov's constant to $1/4$. 
In their work, Li and Ning developed new machinery for Nikiforov's problem, 
which incorporates an even-cycle analog of the Erd\H{o}s--Gallai theorem (see Lemma 2 in \cite{LiNing2023}), 
a hybrid application of the Gould--Haxell--Scott theorem \cite{GouldHaxell2002} and the Voss--Zuluaga theorem \cite{VossZuluaga1977}, 
as well as the Sun--Das spectral inequality \cite{SunDs2020}. This machinery was later utilized by Zhang~\cite{Zhang2024}, 
who also used a graph-decomposition argument (see \cite{FLSY2015}) to further improve the constant to $1/3$. 
On the other hand, Nikiforov's consecutive cycle problem and tools developed are related to other problems in spectral and extremal graph theory, 
see \cite{LGHSXZ2024,ZhangZhao2023}. 

In this paper, we completely settle Problem \ref{prob:main}.

\begin{theorem}\label{thm:main}
Denote $C_0:=\frac12(3-\sqrt5).$
For every $\varepsilon>0$, there exists $n_0=n_0(\varepsilon)$ such that for all $n\ge n_0$,
any $n$-vertex graph $G$ satisfying \eqref{eq:threshold} contains a cycle $C_g$ for every integer length
$3\le g\le(C_0-\varepsilon)n.$
Furthermore, the largest constant in Problem \ref{prob:main} is $C=C_0.$
\end{theorem}

The paper is organized as follows. In Section \ref{Sec:ProofMethods}, we give a skech of proof methods used in this paper.
In Section \ref{sec:prelim}, we state the regularity, 
short-cycle and matching tools and verify the split-graph obstruction. 
In Sections \ref{Subsec:regularpairs} and \ref{Subsec:spectral}, we introduce the definition of regular pairs, 
and the degree form of Regularity Lemma and spectral properties needed, respectively.
In Sections \ref{sec:regular-pair}
and \ref{sec:embedding}, we develop exact paths in one regular pair, 
and prove the prescribed-parity connected-matching lemma.  
The spectral extraction of a connected matching is given in Section \ref{sec:spectral-matching}. We deal with
consecutive even and odd cycles in Sections \ref{sec:even} and \ref{sec:odd}, respectively.  

\section{Proof Methods}\label{Sec:ProofMethods}
The most novelty of this paper is to introduce the spectral version of connected matching method in spectral graph theory.
A \emph{connected matching} is a matching contained in a connected component of a graph.
The connected matching  method was introduced by \L{}uczak \cite{Luczak1999} in the
study of Ramsey-type problems about paths and cycles: to use connected matchings in conjunction with Szemer\'edi’s regularity lemma. 
An important observation due to \L{}uczak is that the problems about monochromatic paths and cycles in complete graphs
can be reduced to ones about monochromatic connected matchings in almost complete graphs by applying the regularity lemma.
Subsequently, this method was systematically formalized by Figaj and \L{}uczak \cite{FigajLuczak2007} and refined by Letzter in \cite{L2022}. 
For surveys on the applications of Regularity Lemma \footnote{As we known, there are rare applications of regularity lemma to spectral graph theory. 
We refer the reader to \cite[Theorem~4.23]{Guiduli1996}.}, we refer to \cite{KomlosSimonovits1996,SS2019}. 
For the connected matching method and developed versions, we refer to \cite{Luczak1999,FigajLuczak2007,L2022}.

Our proof uses a spectral motivation of the connected-matching method.  We apply the
degree form of Szemer\'edi's regularity lemma (Theorem \ref{thm:regularity}) while retaining spectral radius
$n/2-o(n)$ (Lemma \ref{thm:matching-spectrum}).  The reduced graph therefore has spectral radius
$r/2-o(r)$, where $r$ is the number of parts given in the regularity lemma.  The sharp spectral theorem for graphs of given matching number,
due to Feng, Yu, and Zhang \cite{FengYuZhang2007}, then produces a connected
matching of size $(c_0-o(1))r$ (Lemma \ref{lem:matching-extraction}).  

Exact length control is supplied by a self-contained embedding argument.  
A regular pair contains paths of every admissible odd length between prescribed typical endpoints.  
Replacing marked matching edges in a closed walk of the
reduced graph by such paths yields every required even length.  The same
construction gives the odd lengths when the relevant reduced component is non-bipartite (Lemma \ref{lem:bipartite-block}).

The remaining case is substantially different.  
If the component
of the reduced graph is bipartite, then it covers almost all reduced vertices.  
From a spanning tree of this component we form an almost-spanning
bipartite $2$-connected subgraph of the original graph.  A Schur-complement
estimate shows that the block containing this subgraph cannot be bipartite (Lemma \ref{lem:bipartite-block}).
The block therefore contains a short path whose parity is opposite to the
bipartite skeleton.  After deleting the vertices of this path from the
clusters and applying a slicing argument (Lemma \ref{lem:slicing}), a prescribed-endpoint
connected-matching embedding (see Corollary \ref{cor:connected-matching}) produces an interval of odd cycle lengths we wanted.  
Nikiforov's short-cycle theorem covers the bounded initial part, 
this idea already appeared in \cite{LiNing2023}.
  
The key point is that the proof here uses a single large connected
matching to control all lengths simultaneously at the threshold \eqref{eq:threshold}.

\section{Preliminaries}\label{sec:prelim}

For a positive integer $r$, we write $[r]=\{1,\ldots,r\}$.  
Given a graph $G$ and two disjoint nonempty vertex subsets $U,V\subseteq V(G)$, 
we denote by $e(U,V)$ the number of edges of $G$ with one endpoint in $U$ and the other in $V$, 
and define the \emph{edge density} between $U$ and $V$ as
$$d(U,V)=\frac{e(U,V)}{|U||V|}.$$
For $v\in V(G)$ and $U\subseteq V(G)$, 
we write $d_U(v)$ for the number of neighbors of $v$ in $U$.

\subsection{Regular pairs}\label{Subsec:regularpairs}

Given $\eta\in (0,1)$.
A pair $(U,V)$ of disjoint vertex subsets in a graph $G$ is called \emph{$\eta$-regular} if, for all
$X\subseteq U$ and $Y\subseteq V$ with $|X|\ge\eta|U|$ and $|Y|\ge\eta|V|$, 
one has
\[
 \big|d(X,Y)-d(U,V)\big|<\eta.
\]
We use the following degree form of Szemer\'edi's regularity lemma (see \cite{KomlosSimonovits1996, Szemeredi1978}).

\begin{theorem}[Degree form of the regularity lemma]\label{thm:regularity}
  For all $d\in(0,1]$, $\eta\in(0,1)$, and any fixed positive integer $m_0\geq2$, there exist
 integers $M_0=M_0(d,\eta,m_0)$ and $n_0=n_0(d,\eta,m_0)$ 
 such that every graph $G$ of order $n\ge n_0$ admits a partition
 $V(G)=V_0\cup V_1\cup\cdots\cup V_r$
 and a spanning subgraph $G'$ satisfying the following properties:

(i) $m_0\le r\le M_0$;

(ii) $|V_0|\le\eta n$ and $|V_1|=\cdots=|V_r|$;

(iii) $d_{G'}(v)\ge d_G(v)-(d+\eta)n$ for every $v\in V(G)$;

(iv) $G'[V_i]$ is empty for every $i\in[r]$;

(v) For $G'$ and all $1\le i<j\le r$, 
the pair $(V_i,V_j)$ is $\eta$-regular with $d(V_i,V_j)\in \{0\}\cup [d,1]$.
\end{theorem}

Let $G'$ be the graph constructed as in Theorem \ref{thm:regularity}.
The \emph{reduced graph} $R(G')$ of $G'$ is the graph with vertex set $[r]$ and edge set $\{ij: d(V_i,V_j)>0\}$.
Such a $G'$ is called a \emph{host graph}.

\subsection{Useful spectral properties}\label{Subsec:spectral}

We will repeatedly make use of the following two elementary spectral properties. 
Let $\Delta(G)$ be the maximum degree of a graph $G$. Then 
\begin{equation}\label{eq:spectral-basic}
 \rho(G)\le \Delta(G)
\quad\text{and}\quad
 \rho(G)\le\sqrt{2e(G)}.
\end{equation}
The second inequality follows from the basic spectral identity that 
the sum of the squares of all adjacency eigenvalues of $G$ equals $2e(G)$.

Given a positive integer $t$ and a graph $G$, we denote by $G[t]$ the \emph{$t$-blow-up} of $G$,
which is constructed from $G$ via the following two steps:

(i) Replace each vertex $i\in V(G)$ with an independent set $V_i$ of size $t$.

(ii) The edge set of $G[t]$ is given by $E(G[t])=\{v_iv_j:~v_i\in V_i, v_j\in V_j, ij\in E(G)\}$.

\noindent Additionally, we have the following spectral property:
\begin{equation}\label{eq:blowup-spectrum}
 \rho\big(G[t]\big)=t\cdot\rho\big(G\big).
\end{equation}
Equality (\ref{eq:blowup-spectrum}) holds 
because the adjacency matrix of $G[t]$ satisfies $A(G[t])=A(G)\otimes J_t$,
where $J_t$ denotes the $t\times t$ all-ones matrix and $\otimes$ is the Kronecker product.

The following result on spectral consecutive cycles was originally proved by Nikiforov \cite{Nikiforov2008}.

\begin{theorem}[Nikiforov \cite{Nikiforov2008}]\label{thm:short-cycles}
There exists a positive integer $n_1$ such that every graph $G$ of order $n\ge n_1$ satisfying
\eqref{eq:threshold} contains a cycle $C_g$ for every integer length
$3\le g\le n/320.$
\end{theorem}

Denote by $S_{n,k}$ the split graph $K_k\vee\overline K_{n-k}$.
We can partition its vertex set as $\Pi: V(S_{n,k})=U\cup V$,
where $U$ is the $k$-clique and $V$ is the independent set of size $n-k$.
Clearly, the partition $\Pi$ is equitable, and thus
$\rho(S_{n,k})$ is the largest eigenvalue of the quotient matrix
\[
B_{\Pi} = \begin{bmatrix}
k-1 & n-k \\
k & 0
\end{bmatrix}.
\]
Thus, $\rho(S_{n,k})$ is the largest root of the characteristic polynomial:
\[
f(x) :=\det(xI_2-B_{\Pi})=x^2-(k-1)x-k(n-k).
\]
Consequently,
\begin{equation}\label{eq:split-spectrum}
\rho(S_{n,k})=\frac12\Big(k-1+\sqrt{4kn-3k^2-2k+1}\Big).
\end{equation}
Define $\phi(x):=\frac12\big(x+\sqrt{4x-3x^2}\big)$, where $0\le x\le1$.
We immediately obtain
\begin{equation}\label{eq:split-asymptotic}
 \frac{\rho(S_{n,k})}{n}\longrightarrow\phi(x), \quad \text{whenever} ~~ \frac kn\to x.
\end{equation}

\begin{theorem}\label{prop:sharpness}
The constant $C_0$ in Theorem \ref{thm:main} is the best possible.
\end{theorem}

\begin{proof}
Recall that $C_0=\frac12(3-\sqrt5).$ 
Suppose to the contrary that $C_0$ is not the best possible.
Then, $C>C_0$ for the maximum constant $C$ in Problem \ref{prob:main}. 
It is easy to check that 
$$\phi'(x)=\frac12+\frac{2-3x}{2\sqrt{4x-3x^2}}>0$$
for $x\in (0,\frac12]$.
Thus, $\phi(x)$ is strictly increasing on $[0,\frac12]$.

Choose a real number $a$ such that
$\frac12C_0<a<\min\{\frac12C,\frac12\}$, and set $k=\lfloor{an}\rfloor$.  
Note that $\phi(\frac12 C_0)=\frac12$.
By \eqref{eq:split-asymptotic} and the monotonicity of $\phi(x)$, 
there exists a constant $\gamma>0$ such that
$\rho(S_{n,k})>(\frac12+\gamma)n$ for all sufficiently large $n$.  
In particular, $\rho(S_{n,k})>\sqrt{\lfloor{n^2/4}\rfloor}$.

Since $k<\frac12n$, every cycle in $S_{n,k}$ contains at most $k$ vertices of
its independent set and thus has length at most $2k$.  
Moreover, $S_{n,k}$ contains a cycle of length $2k$.
Hence, its circumference is precisely $2k$.
Choose $\varepsilon=\frac12(C-2a)>0$.  
Then $2k+2\le2an+2\leq(C-\varepsilon)n$ for all sufficiently large $n$,
but $S_{n,k}$ does not contain a copy of $C_{2k+2}$.  
This contradicts the definition of $C$ in Problem \ref{prob:main}.
Therefore, we conclude that $C=C_0$.
\end{proof}

\subsection{Exact paths in a regular pair}\label{sec:regular-pair}

Let $(U,V)$ be an $\eta$-regular pair with edge density $p$.  
Let $X\subseteq U$ and $Y\subseteq V$ satisfy $|X|\geq\eta|U|$ and $|Y|\geq\eta|V|$.
A vertex $x\in X$ is said to be \emph{typical} into $Y$, if $d_Y(x)\geq(p-\eta)|Y|$.
In turn, we say the subset $X$ is \emph{typical} into $Y$, 
if every vertex $x\in X$ is typical into $Y$.
We first state the following lemma, which is a standard consequence of regularity lemma.

\begin{lemma}\label{lem:typical}
Let $(U,V)$ be an $\eta$-regular pair with edge density $p$.  
For any subset $Y\subseteq V$ with $|Y|\ge\eta|V|$, 
the number of vertices $x\in U$ satisfying $d_Y(x)<(p-\eta)|Y|$ is fewer than $\eta|U|$.
The analogous statement with $U$ and $V$ interchanged also holds.
\end{lemma}

\begin{proof}
Otherwise, choose a subset $X\subseteq U$ of size at least $\eta|U|$,
such that every vertex $x\in X$ satisfies $d_Y(x)<(p-\eta)|Y|$.  
Then $d(X,Y)<p-\eta$, 
and thus $d(U,V)-d(X,Y)>\eta$.
This contradicts the $\eta$-regularity of $(U,V)$.
\end{proof}

The following lemma enables precise control over the length of the target structure, 
while permitting a bounded set of vertices that have been used in prior steps.

\begin{lemma}\label{lem:flex-path}
Fix parameters $0<d\leq 1$, $0<\eta\le d^2/1000$, and a positive integer $k$.  
There exists a threshold $\ell_0=\ell_0(d,\eta,k)$ such that the following holds.

Suppose that a graph $G$ contains an $\eta$-regular pair $(U,V)$ with edge density $p\ge d$, 
where both sets have size $\ell\ge \ell_0$.
Let $W$ be a subset of $U\cup V$ satisfying 
$|W\cap U|\le k$ and $|W\cap V|\le k.$
Take $x\in U\setminus W$ and $y\in V\setminus W$ such that
$d_V(x)\ge(p-\eta)\ell$ and $d_U(y)\ge(p-\eta)\ell.$

Define
\begin{equation}\label{eq:S-def}
   s^*=s^*(d,\eta,k,\ell):=\Big\lfloor{\big(1-\frac d4-\frac{10\eta}{d}\big)\ell\Big\rfloor}.
\end{equation}
Then for every choice of $s\in [s^*]$, there exists an
$x$--$y$ path of length $2s+1$ in $G$ whose vertices all belong to the set $(U\cup V)\setminus W$.
\end{lemma}

\begin{proof}
Choose $\ell_0$ sufficiently large in terms of $d$, $\eta$, and $k$. Since
Since $d_U(y)\ge(p-\eta)\ell$ and $p\ge d>1000\eta$, we can choose a vertex subset 
$Z\subseteq N_G(y)\cap\bigl(U\setminus(W\cup\{x\})\bigr)$ with 
$|Z|=\lfloor{d\ell/4}\rfloor.$
For sufficiently large $\ell$, we have $|Z|\ge\eta\ell$. 

Set $ U_0:=U\setminus(W\cup Z\cup\{x\})$
and $V_0:=V\setminus(W\cup\{y\}).$
Next, we shall choose distinct vertices
$u_1,\ldots,u_{s-1}\in U_0$, $v_0,\ldots,v_{s-1}\in V_0$, and $u_s\in Z$
so that
 \begin{equation}\label{eq:flex-path-form}
   x v_0 u_1 v_1\cdots u_{s-1}v_{s-1}u_s y
 \end{equation}
 is a path in $G$.

During the construction, let $U_i=U_0\setminus\{u_1,\ldots,u_i\}$ 
and $V_i=V_0\setminus\{v_0,\ldots,v_{i-1}\}$ for each $i\in\{0,\ldots,s-1\}$.
Since $s\leq s^*$ and $\ell$ is sufficiently large,
we have
 \begin{equation}\label{eq:remaining-sets}
   |U_i|\ge \frac{8}{d}\eta\ell
   \quad \text{and} \quad 
   |V_i|\ge \frac d5\ell.
 \end{equation}
Indeed, at the last possible step the first set has size at least
 $10\eta\ell/d-k-O(1)$, while the second has size at least
 $d\ell/4+10\eta\ell/d-k-O(1)$.

{\bf {(i) Initial step of the construction}}.
Choose $v_0$.
Note that both $U_0$ and $Z$ are of sizes at least $\eta\ell$.
By Lemma \ref{lem:typical}, fewer than $\eta\ell$ vertices $v\in V$ are atypical with respect to $U_0$
 (i.e., satisfy $d_{U_0}(v)<(p-\eta)|U_0|$),
 and fewer than $\eta\ell$ vertices of $V$ are atypical with respect to $Z$.
On the other hand, the endpoint degree assumption $d_V(x)\ge(p-\eta)\ell$ implies
$d_{V_0}(x)\ge(p-\eta)\ell-k-1>2\eta\ell.$
Hence, there exists $v_0\in N_G(u_0)\cap V_0$, which is simultaneously typical into $U_0$ and $Z$.  Thus,
 \begin{equation}\label{eq:ai-good}
   d_{U_0}(v_0)\ge(p-\eta)|U_0|
   \quad \text{and} \quad 
   d_Z(v_0)\ge(p-\eta)|Z|.
 \end{equation}

{\bf (ii) Iterative step.} Choose $u_i$ and $v_i$ for all $i\in [s-1]$.
Now let $i\in [s-1],$ and suppose that $v_{i-1}\in N_G(u_{i-1})\cap V_{i-1}$ has been chosen, 
which is simultaneously typical into $U_{i-1}$ and $Z$.
 Thus, 
 \begin{equation}\label{eq:di-good}
   d_{U_{i-1}}(v_{i-1})\ge(p-\eta)|U_{i-1}|
   \quad \text{and} \quad 
   d_Z(v_{i-1})\ge(p-\eta)|Z|.
 \end{equation}
Note that $|V_{i-1}|\geq\frac d5\ell\geq\eta\ell$. By Lemma \ref{lem:typical}, 
fewer than $\eta\ell$ vertices $u\in U$ are atypical with respect to $V_{i-1}$. 
On the other hand, combining \eqref{eq:di-good} and \eqref{eq:remaining-sets}
gives $d_{U_{i-1}}(v_{i-1})\ge(p-\eta)\cdot\frac{8}{d}\eta\ell>\eta\ell$. 
Hence, there exists $u_i\in N_G(v_{i-1})\cap U_{i-1}$,
which is typical into $V_{i-1}$.  Consequently,
 \begin{equation*}
   d_{V_{i-1}}(u_i)\ge(p-\eta)|V_{i-1}|\ge\big(p-\eta\big)\frac d5\ell>2\eta\ell.
 \end{equation*}
Note that both $U_i$ and $Z$ are of sizes at least $\eta\ell$.
Combining Lemma \ref{lem:typical}
with analogous argument applied to \eqref{eq:ai-good}, 
we can assert the existence of a vertex $v_{i}\in N_G(u_i)\cap V_i$ 
that is simultaneously typical into $U_{i}$ and $Z$.  
It follows that 
 \begin{equation}\label{eq:bi-good}
   d_{U_i}(v_i)\ge(p-\eta)|U_i|
   \quad \text{and} \quad 
   d_Z(v_i)\ge(p-\eta)|Z|.
 \end{equation}

{\bf (iii) The final step of the path construction.} Choose $u_s.$
In view of \eqref{eq:bi-good}, we have $N_G(v_{s-1})\cap Z\ne\varnothing$.  
 Choose $u_s\in N_G(v_{s-1})\cap Z$.  
 Since $Z\subseteq N_G(y)$, the sequence
 \eqref{eq:flex-path-form} is an $x$--$y$ path of length $2s+1$ in $G$.
 Moreover, all of its vertices lie in $(U\cup V)\setminus W$.
\end{proof}

When building our target subgraph iteratively, 
we will repeatedly need to restrict to subsets of initially identified regular pairs 
while retaining enough regularity to continue the argument. 
The following slicing lemma guarantees this restriction is always possible, 
as long as the extracted subsets are sufficiently large relative to the original pair. 

The following lemma is standard. We include the proof here for completeness.

\begin{lemma}[Slicing]\label{lem:slicing}
Let $\eta\in(0,\frac14]$, and let $(U,V)$ be an $\eta$-regular pair with $d(U,V)\ge d>\eta.$ 
Suppose that $U'\subseteq U$ and $V'\subseteq V$ satisfy
$|U'|\ge(1-2\eta)|U|$ and $|V'|\ge(1-2\eta)|V|.$
Then, $(U',V')$ is a $2\eta$-regular pair with $d(U',V')\ge d-\eta.$
\end{lemma}

\begin{proof}
For $0<\eta\leq\frac14$, the constraints on $U'$ and $V'$ imply $|U'|\ge\eta|U|$ and $|V'|\ge\eta|V|$.
The $\eta$-regularity of $(U,V)$ then yields
$|d(U',V')-d(U,V)|<\eta,$ and hence $d(U',V')\ge d-\eta$.  
 
Let $X\subseteq U'$ and $Y\subseteq V'$  such that
$|X|\ge2\eta|U'|$ and $|Y|\ge2\eta|V'|.$
 Since $2\eta(1-2\eta)\ge\eta$, 
 we obtain $|X|\ge\eta|U|$ and $|Y|\ge\eta|V|$. 
The $\eta$-regularity of $(U,V)$ also gives
$|d(X,Y)-d(U,V)|<\eta.$
Consequently, we deduce that
$$\big|d(X,Y)-d(U',V')\big|\leq\big|d(X,Y)-d(U,V)\big|+\big|d(U',V')-d(U,V)\big|<2\eta.$$
Therefore, $(U',V')$ is $2\eta$-regular.
\end{proof}

\subsection{A new connected-matching embedding lemma}\label{sec:embedding}

A \emph{connected matching} of a graph $G$ is defined as a matching 
whose edges are all contained in a single connected component of $G$. 
A graph $G$ is said to be \emph{Eulerian},
if it admits an Euler tour, which is a closed walk that traverses every edge of $G$ exactly once.
A connected graph is 
Eulerian if and only if all vertices have even degrees.

The next lemma may be viewed as an exact-length refinement 
of the connected-matching embedding technique initiated by \L{}uczak \cite{Luczak1999} 
and developed by Figaj and \L{}uczak  \cite{FigajLuczak2007}. Unlike the previous applications, 
where the connected matching method was used to obtain cycles of prescribed asymptotic size, 
we require precise control of every admissible length and parity.

\begin{lemma}[Connected-matching embedding lemma]\label{lem:closed-walk}
Fix $0<d\leq 1$, $0<\eta\le d^2/1000$, and a positive integer $k$.  
There exists a threshold $\ell_1=\ell_1(d,\eta,k)$ such that the following holds.

Let $H$ be a graph with vertex set $[r]$ that admits a connected matching $M$ of size $t\ge1$.
Suppose that a graph $G$ contains pairwise disjoint vertex subsets $V_1,\ldots,V_r$,
each of size $\ell\ge \ell_1,$
such that for any $ab\in E(H)$, 
the pair $(V_a,V_b)$ is $\eta$-regular in $G$, with edge density $p_{ab}\geq d$.

Let $H_0$ be a connected subgraph of $H$ with $M\subseteq E(H_0)$, 
and let $H^*$ be a loopless Eulerian multigraph of size $m\leq k$ 
obtained from $H_0$ by adding parallel edges to some edges in $E(H_0)$.

Then for every choice of $s_1,\ldots, s_t\in [s^*]$,  
where $s^*$ is defined by \eqref{eq:S-def}, the host graph
$G$ contains a cycle of length $m+2\sum_{j=1}^{t}s_j$,
and thus cycles of every length $g$ satisfying $g\equiv m\pmod2$ and $m+2t\le g\le m+2ts^*.$
\end{lemma}

\begin{proof}
Choose $\ell_1\geq\ell_0$ sufficiently large so that all greedy lower bounds below are positive. 
Take an Euler tour $W$ of $H^*$ and mark the edge occurrences corresponding to $M$.
Two marked occurrences
cannot be consecutive in $W$, because consecutive edges of a walk share a vertex whereas the marked
edges form a matching. Deleting the $t$ marked occurrences from the cyclic edge sequence of $W$ therefore
leaves $t$ nonempty walk segments.

We embed these unmarked connector segments of $W$ as pairwise vertex-disjoint paths in $G$. 
For every occurrence $a_h$ of $W$, 
we shall first define a candidate set $X_{h}\subseteq V_{a_h}$ as follows.
If the occurrence $a_h$ is an endpoint of some connector segment, 
it must admits a deleted marked edge $a_hb$. We then define
$$X_{h}:=\big\{v\in V_{a_h}:~ d_{V_b}(v)\geq(p_{a_hb}-\eta)\ell\big\},$$
i.e., every vertex in $X_{h}$ is typical into $V_b.$
If no marked edge is adjacent to this vertex occurrence, then we set $X_{h}=V_{a_h}$. 
Because marked edge occurrences are not
consecutive, every vertex occurrence is subject to at most one external endpoint condition. 
Lemma \ref{lem:typical} implies 
 \begin{equation}\label{eq:ah}
  |X_{h}|\geq(1-\eta)\ell.
 \end{equation}

Having chosen a candidate set $X_h \subseteq U_{a_h}$ for each occurrence $a_h$ of $W$, 
we then fix an anchor for $a_h$ from $X_h$, avoiding all vertices used as anchors in prior steps.
Since the total number of vertex occurrences in the Euler tour $W$ satisfies $m\le k$,
at most $k$ vertices are forbidden in any cluster $V_i$ of $G$ at each step of the embedding. 

For a present segment $a_0a_1\ldots a_q$, choose an unused anchor $v_0\in X_{0}$ for $a_0$ 
so that it is typical into $X_{1}.$ 
This is possible because $\min\{|X_{0}|,|X_{1}|\}\geq(1-\eta)\ell>\eta\ell+k$ and, by Lemma \ref{lem:typical}, 
fewer than $\eta\ell$ vertices of $V_{a_0}$ are atypical with respect to $X_1.$

Since $v_0\in X_{0}$ is typical into $X_1$, it follows from (\ref{eq:ah}) that 
$$d_{X_1}(v_0)\geq (p_{a_0a_1}-\eta)|X_1|\geq(d-\eta)(1-\eta)\ell>\eta\ell+k.$$
Furthermore, by Lemma \ref{lem:typical},
fewer than $\eta\ell$ vertices of $V_{a_1}$ are atypical with respect
to $X_2.$
Therefore,
we can choose an unused anchor $v_1\in N_G(v_0)\cap X_{1}$ for $a_1$ that is typical into $X_{2}.$ 

Suppose that for all $h\leq q-2$, an unused anchor $v_h\in N_G(v_{h-1})\cap X_{h}$ for $a_h$ has been chosen to be typical into $X_{h+1}$.  
Then we have
$$d_{X_{h+1}}(v_h)\geq(p_{a_ha_{h+1}}-\eta)|X_{h+1}|\geq(d-\eta)(1-\eta)\ell>\eta\ell+k.$$
By Lemma \ref{lem:typical}, fewer than $\eta\ell$ vertices of $V_{a_{h+1}}$ are atypical with respect to
 $X_{h+2}$. 
Therefore,
we can choose an unused anchor $v_{h+1}\in N_G(v_h)\cap X_{h+1}$ for $a_{h+1}$ that is typical into $X_{h+2}.$ 

In the final step of the current segment, the set $N_G(v_{q-1})\cap X_q$ contains at least $(d-\eta)(1-\eta)\ell$ vertices, so it
must include an unused anchor $v_q$ for $a_q$. Thus, we successfully embed this segment into $G$ as the path $v_0v_1\ldots v_q$.
Repeating the argument embeds all $t$ connector
segments into $G$ as pairwise vertex-disjoint paths.

On the other hand, 
by the definition of the endpoint candidate sets, if $ab$ is a marked edge, and its two
connector endpoints are embedded into $G$, with one anchored at $x\in V_a$ and the other at $y\in V_b$, then
$d_{V_b}(x)\geq(p_{ab}-\eta)\ell$ and $d_{V_a}(y)\geq(p_{ab}-\eta)\ell$.
Let $W_{ab}$ denote the set of all other fixed anchors chosen from $V_a\cup V_b$. 
Since the Euler tour $W$ has at most $k$ vertex occurrences, we have
$|W_{ab}\cap V_a|\leq k$ and $|W_{ab}\cap V_b|\leq k$.
Apply Lemma \ref{lem:flex-path} to $(V_a,V_b)$ with forbidden set $W_{ab}$ and prescribed integer $s_j$. 
This replaces the marked edge $ab$ by an $x-y$ path of length $2s_j+1$ 
whose vertices all lie in $(V_a\cup V_b)\setminus W_{ab}$. 
Since the $t$ marked edges form a matching, 
the $t$ replacement paths lie in disjoint pairs of clusters and are mutually vertex-disjoint.

Following the cyclic order of the Euler tour $W$, the $t$ embedding paths for unmarked connector segments 
and the $t$ replacement paths for marked edges together form a cycle. 
The embedding paths contribute exactly $m-t$ edges, 
while the replacement paths contribute $\sum_{j=1}^t(2s_j+1)$ edges.
The total length of the cycle is thus $m+\sum_{j=1}^t2s_j$.
Finally, every in the interval integer $[t,ts^*]$ can be expressed as a sum of $t$ integers in $[s^*]$, 
which gives the asserted interval.
\end{proof}

Using Lemma \ref{lem:closed-walk}, 
we can construct both even and odd cycles of consecutive lengths. 
The parity of such cycles is completely determined 
by the parity of the reduced connected subgraph $H_0$ that contains the matching $M$.

\begin{corollary}\label{cor:connected-matching}
 Let $H$ and $G$ satisfy the assumptions of Lemma \ref{lem:closed-walk}.    
 If $\ell\ge \ell_1(d,\eta,2r)$, where $r\leq s^*+1,$
 then $G$ contains a cycle $C_g$ for every even $g\in\{4,\ldots,2ts^*\}.$
\end{corollary}

\begin{proof}
 Let $M=\{e_1,\ldots,e_t\}$ be a connected matching of $H$,
 i.e., all edges of $M$ lie in the same component $H_0$ of $H$.
Define $M_j:=\{e_1,\ldots,e_j\}$ for $j\in [t]$.
Since the subgraph of $H_0$ induced by $M_j$ is a forest, 
there exists a spanning tree $T_j$ of $H_0$ such that $M_j\subseteq E(T_j)$.
For $j=1$, we may take $T_1$ to be the single edge $e_1$.
Note that $n(T_j)\leq n(H)=r$. 
Then $e(T_j)\le r-1$ for each $j\in [t]$.

For each $j\in [t]$, let $T_j^*$ denote the multigraph obtained by replacing every edge of $T_j$ with two parallel copies.
Then $T^*_j$ is Eulerian,
with $e(T^*_j)=2e(T_j)<2r$.
As a consequence of Lemma \ref{lem:closed-walk}, 
for every even integer $g$ belonging to the interval $I_j=[\,2e(T_j)+2j,\;2e(T_j)+2js^*\,],$
the graph $G$ contains a cycle $C_g$ of length $g$. 
In particular, when $j=1$, this interval simplifies to $[4,2+2s^*]$.

 The intervals $I_j$ for $j\in [t]$ overlap such that there are no missing even integers in their union.
 Indeed, since $j\leq e(T_j)\le r-1\leq s^*$, for $j\geq 2$ we can deduce that
 \begin{align*}
  \max I_{j-1}+2
\ge 2(j-1)+2(j-1)s^*+2
  \ge2j+2(r-1)
  \ge\min I_j.
 \end{align*}
 Finally, $\max I_t\ge2ts^*$.  
Therefore, $G$ contains a cycle $C_g$ for every even $g\in\{4,\ldots,2ts^*\}.$
\end{proof}

\begin{corollary}\label{cor:nonbipartite-odd} Let $H$ and $G$ satisfy the assumptions of
Lemma \ref{lem:closed-walk}.  
Suppose $H$ admits a connected non-bipartite subgraph $H_0$, which contains a matching $M$ of size $t$.  
If $\ell\ge\ell_1(d,\eta,2r),$ where $r\leq s^*+1$,
then $G$ contains a cycle $C_g$ for every odd $g\in \{2r,\ldots, 2ts^*\}$.
\end{corollary}

\begin{proof} 
Let $M=\{e_1,\ldots,e_t\}$ be a matching of $H_0$. Define $M_j:=\{e_1,\ldots,e_j\}$ for $j\in [t]$.
Fix an odd cycle $C$ in $H_0$.
Building on $C$ and $M_j$,
one can easily construct a connected unicyclic subgraph $T_j\subseteq H_0$ such that $\big(M_j\cup E(C)\big)\subseteq E(T_j).$
For $j=1$, we set $T_1=C$.
Note that $n(T_j)\leq n(H)=r$. 
Then $e(T_j)\le n(T_j)\le r$ for each $j\in [t]$. 
Moreover, since there are at least $\frac12(|C|+1)$ edges in $C$ do not belong to any matching, 
we have $e(T_j)\geq |M_j\cup E(C)|\geq j+\frac12(|C|+1)$
for each $j\in [t]$.

For each $j\in [t]$, let $T_j^*$ denote the multigraph obtained from $T_j$ 
by replacing every edge in $E(T_j)\setminus E(C)$ with two parallel copies. Then $T^*_j$ is Eulerian,
with $e(T^*_j)=2e(T_j)-|C|<2r-2$.
As a consequence of Lemma \ref{lem:closed-walk}, for every odd integer $g$ 
belonging to the interval $$I_j=\big[\,2e(T_j)-|C|+2j,\;2e(T_j)-|C|+2js^*\,\big],$$ the graph $G$ contains a cycle $C_g$ of length $g$. 
In particular, for $j=1$, we have $e(T_1)=|C|$, and the corresponding interval reduces to $[\,|C|+2,|C|+2s^*\,]$.

 The intervals $I_j$ for $j\in [t]$ overlap such that there are no missing odd integers in their union.
 Indeed, since $j+\frac12(|C|+1)\leq e(T_j)\le r\leq s^*+1$, for $j\geq 2$ we can derive that
 \begin{align*}
  \max I_{j-1}+2
\ge 2(j-1)+2(j-1)s^*+3
  \ge 2j+1+2(r-1)
  \ge\min I_j.
 \end{align*}
Finally, $\max I_t\ge2ts^*$,  
so $G$ contains a cycle $C_g$ for every odd $g\in\{|C|+2,\ldots,2ts^*\}.$
Since $3\leq|C|\leq n(H_0)\leq n(H)=r$, this holds for all odd $g\in\{2r,\ldots,2ts^*\}.$
\end{proof}

\subsection{Extraction of a connected matching under a spectral condition}\label{sec:spectral-matching}

The following result is a direct consequence of the sharp theorem of Feng, Yu, and Zhang \cite{FengYuZhang2007}
on the spectral radius of a graph with prescribed matching number.

\begin{lemma}\label{thm:matching-spectrum}
Let $G$ be a graph of order $n$ with matching number $\nu$, where $n\geq2\nu$. Then
 \begin{equation}\label{eq:matching-spectrum}
  \rho(G)\le
  \max\Big\{
    2\nu,
    \frac12\Big(\nu-1+\sqrt{(\nu-1)^2+4\nu(n-\nu)}\Big)\Big\}.
 \end{equation}
\end{lemma}

The two terms in \eqref{eq:matching-spectrum} correspond precisely to the two extremal constructions:
the complete graph $K_{2\nu+1}$ plus $(n-2\nu-1)$ isolated vertices, 
and the split graph $S_{n,\nu}=K_\nu\vee\overline{K_{n-\nu}}$. 

Recall that a connected matching of a graph $G$ 
is defined as a matching all of whose edges lie in the same connected component of $G$.

\begin{lemma}\label{lem:matching-extraction}
 Let $c_0:=\frac14(3-\sqrt5)$.
 For every $\alpha\in(0,c_0)$, there exists $\delta=\delta(\alpha)>0$ such
 that every graph $G$ of order $n$ satisfying
$\rho(G)>(\frac12-\delta)n$
 has a connected matching of size at least $(c_0-\alpha)n$.
\end{lemma}

\begin{proof}
Let $G_0$ be a connected component of $G$ with $\rho(G_0)=\rho(G)$, 
and let $\nu_0=\nu(G_0)$.  
Define
$\phi(x)=\frac12\big(x+\sqrt{4x-3x^2}\big).$
Following the arguments in the proof of Theorem \ref{prop:sharpness}, 
we obtain $\phi(c_0)=\frac12$ and that $\phi$ is strictly increasing on the interval $[0,c_0]$.

Suppose, for the sake of contradiction, that $\nu_0<(c_0-\alpha)n$.  
Applying Lemma \ref{thm:matching-spectrum} to $G_0$, which has order at most $n$, 
yields the upper bound $\rho(G_0)\le\max\big\{2\nu_0, n\phi(\nu_0/n)\big\}$.
Indeed, the second term in \eqref{eq:matching-spectrum} is at most
$$\frac12\Big(\nu_0+\sqrt{\nu_0^2+4\nu_0(n-\nu_0)}\Big)=n\phi\Big(\frac{\nu_0}n\Big).$$
Hence, $\rho(G_0)\le\max\big\{2(c_0-\alpha),\phi(c_0-\alpha)\big\}\,n.$
Both quantities inside the maximum are strictly smaller than $\frac12$.
Choosing $\delta>0$ so that
$\max\big\{2(c_0-\alpha),\phi(c_0-\alpha)\big\}
  \le\frac12-2\delta$
 yields $\rho(G_0)\leq(\frac12-2\delta)n$.
 This leads to a contradiction.  
 Thus, $\nu_0\ge(c_0-\alpha)n$, as desired.
\end{proof}

\section{Consecutive even cycles}\label{sec:even}

In this section, we establish the even-cycle case of Theorem \ref{thm:main}.

\begin{theorem}\label{thm:even}
 For every $\varepsilon>0$, there exists an integer $n_0=n_0(\varepsilon)$ such that every graph $G$
 of order $n\ge n_0$ satisfying \eqref{eq:threshold} contains a cycle $C_g$ for
 every even length $4\le g\le(C_0-\varepsilon)n$.
\end{theorem}

\begin{proof}
Recall that $C_0=\frac12(3-\sqrt5)$ and $c_0=\frac12C_0$.
We may assume that $\varepsilon<C_0$, since otherwise the assertion is vacuous.  
Choose positive constants $\alpha,\theta,\zeta$ so
 small that
 \begin{equation}\label{eq:parameter-main}
   2(c_0-\alpha)(1-\theta)(1-\zeta)>2c_0-\varepsilon.
 \end{equation}

 Let $\delta=\delta(\alpha)$ be the constant supplied by Lemma \ref{lem:matching-extraction}.  
 We first choose $0<d<\min\{\frac \delta4, 2\theta, 2\varepsilon,1\}$, 
 and then choose $\eta>0$ sufficiently small so that all of the following constraints hold:
 \begin{equation}\label{eq:parameter-reg}
 \eta\le\frac{d^2}{1000},
  \quad
  \eta<\zeta,
  \quad
  \frac d4+\frac{10\eta}{d}<\min\{\theta,\varepsilon\},
  \quad
  d+\eta+\sqrt{2\eta}<\delta.
 \end{equation}
 
 We apply Theorem \ref{thm:regularity} to $G$
 with the aforementioned parameters and an arbitrary fixed integer $m_0\ge2$.  
The theorem yields a vertex partition $V(G)=V_0\cup V_1\cup\cdots\cup V_r$  
and the corresponding spanning subgraph $G'$, 
where $V_0$ is the exceptional set of size at most $\eta n$, 
and $V_1,\ldots,V_r$ are the nonexceptional regular clusters of equal size $\ell$.

Let $G''$ be the spanning subgraph of $G$ with $E(G'')=E(G)\setminus E(G')$. 
By the degree conclusion in Theorem \ref{thm:regularity}, we have $\Delta(G'')\le(d+\eta)n.$ 
In view of (\ref{eq:spectral-basic}), we know that $\rho(G'')\leq\Delta(G'')\le(d+\eta)n.$
Since $A(G)=A(G')+A(G'')$ and the spectral radius of real symmetric matrices is subadditive under matrix addition, 
it follows that
 \begin{equation}\label{eq:rho-Gprime}
 \rho(G')\ge\rho(G)-\rho(G'')\geq \rho(G)-(d+\eta)n.
 \end{equation}

Let $H'$ and $H''$ be two spanning subgraphs of $G'$,
where $E(H')=E(G'-V_0)$ and $E(H'')$ consists of all edges incident to vertices in $V_0$.  
Since $e(H'')\le|V_0|n\le \eta n^2$, 
invoking \eqref{eq:spectral-basic} yields $\rho(H'')\leq\sqrt{2e(H'')}\le\sqrt{2\eta}\,n.$
Note that $A(G')=A(H')+A(H'')$. 
In light of \eqref{eq:rho-Gprime}, we obtain
 \begin{equation}\label{eq:rho-H}
  \rho(H')\geq \rho(G')-\rho(H'')
  \ge\rho(G)-\bigl(d+\eta+\sqrt{2\eta}\bigr)n.
 \end{equation}
Note that $\sqrt{\lfloor{n^2/4}\rfloor}=\frac 12n+O(n^{-1}).$
Thus, \eqref{eq:threshold}, \eqref{eq:parameter-reg}, and
 \eqref{eq:rho-H} imply, for all sufficiently large $n$, that
 \begin{equation}\label{eq:rho-H-lower}
   \rho(H')>\sqrt{\lfloor{n^2/4}\rfloor}-\bigl(d+\eta+\sqrt{2\eta}\bigr)n>\big(\frac12-\delta\big)n.
 \end{equation}

Let $R(G')$ be the reduced graph associated with $G'$,
which is the graph with vertex set $[r]$ and edge set $\{ij: d(V_i,V_j)>0\}$.
Clearly, $H'$ is a subgraph of the $\ell$-blow-up of $R(G')$.  
In view of \eqref{eq:blowup-spectrum} and \eqref{eq:rho-H-lower}, we obtain
 \begin{equation}\label{eq:rho-H-odd}
   \ell\cdot \rho(R(G'))\ge\rho(H')>
  \big(\frac12-\delta\big)n
   \ge \big(\frac12-\delta\big)r\ell.
 \end{equation}
Thus, $\rho(R(G'))>(\frac12-\delta)r.$
By Lemma \ref{lem:matching-extraction}, 
$R(G')$ has a connected component, which contains a matching of size $t\ge(c_0-\alpha)r.$

By conclusion (v) of Theorem \ref{thm:regularity}, 
every edge $ij$ of $R(G')$ represents an $\eta$-regular pair $(V_i,V_j)$ of edge density at least $d$.  
Define $s^*=\lfloor{(1-\frac d4-\frac{10\eta}{d})\ell}\rfloor$ per \eqref{eq:S-def}.  
Conclusion (i) of Theorem \ref{thm:regularity} bounds $r$ above by
 a constant $M_0=M_0(d,\eta,m_0)$ independent of $n$, while $\ell\to\infty$ as $n$ grows.  
 Hence, for sufficiently large $n$, we have both
$\ell\ge \ell_1(d,\eta,2r)$ and $s^*\ge r-1.$
 All hypotheses of Corollary \ref{cor:connected-matching} therefore hold, and $G$
 contains a cycle $C_g$ for every even length $4\le g\le2ts^*$.

By the third inequality in \eqref{eq:parameter-reg}, we have
$1-\frac d4-\frac{10\eta}{d}>1-\theta$. Therefore, $s^*\ge(1-\theta)\ell-1.$
Recall that $t\ge(c_0-\alpha)r$. 
Together with  $r\ell=n-|V_0|\ge(1-\eta)n\ge(1-\zeta)n$, this gives
 \begin{align*}
 2ts^*\ge2(c_0-\alpha)r\cdot\bigl((1-\theta)\ell-1\bigr)\ge2(c_0-\alpha)(1-\theta)(1-\zeta)n-O(r).
 \end{align*}
 Since $r$ is bounded independently of $n$, 
 inequality \eqref{eq:parameter-main} implies that
$2ts^*>(2c_0-\varepsilon)n=(C_0-\varepsilon)n$ for all sufficiently large $n$.  
Hence, $G$ contains a cycle $C_g$ for every even length
$4\le g\le(C_0-\varepsilon)n.$ 
This completes the proof.
\end{proof}

\section{Consecutive odd cycles}\label{sec:odd}

Building on the even-cycle analysis in Section \ref{sec:even}, 
we now establish the odd-cycle case of Theorem \ref{thm:main}.
The argument adopts the same connected matching results established in Sections \ref{sec:embedding} and \ref{sec:spectral-matching}. 
In Lemma \ref{lem:closed-walk}, if the reduced connected subgraph $H_0$ 
that contains the matching $M$ is non-bipartite, we can construct an Eulerian multigraph $H^*$ of odd size $m$ 
by fixing an odd cycle in $H_0$ and doubling all remaining edges of $H_0$.
By Corollary \ref{cor:nonbipartite-odd},  
the Euler tour of $H^*$ 
yields all odd cycles of sufficiently large length. The main work is therefore the bipartite reduced case.

\subsection{A bipartite regular skeleton}

Fix parameters $0<d\leq 1$ and $0<\eta\le d^2/1000$, 
and take $\ell$ to be sufficiently large.
Let $T$ be a tree with vertex set $[r]$, where $r\geq 2$.
Suppose that a graph $G$ contains pairwise disjoint vertex subsets $V_1,\ldots,V_r$,
each of size $\ell,$
such that every pair $(V_i,V_j)$ is $\eta$-regular in $G$ with edge density $p_{ij}$,
where $p_{ij}\geq d$ when $ij\in E(T)$ and  $p_{ij}=0$ otherwise.

Given an arbitrary edge $ij\in E(T)$.
Define
$V_i(j):=\{\,x\in V_i:d_{V_j}(x)\ge(p_{ij}-\eta)\ell\,\}$,
i.e., every vertex in $V_i(j)$ is typical into 
$V_j$.
Furthermore, define $B_{ij}$ as the bipartite subgraph of $G$ 
with vertex set $V_i(j)\cup V_j(i)$ and edge set consisting of all edges between $V_i(j)$ and $V_j(i).$
By Lemma \ref{lem:typical}, for every edge $ij\in E(T)$ we have
\begin{equation}\label{eq:pair-core-size}
 \min\Big\{\big|V_i(j)\big|,\big|V_j(i)\big|\Big\}
 \ge\big(1-\eta\big)\ell.
\end{equation}

\begin{lemma}\label{lem:pair-core} 
Assume $T$ and $G$ satisfy the hypotheses above. Then for every edge $ij\in E(T)$, 
the subgraph $B_{ij}$ is $2$-connected.
\end{lemma}

\begin{proof}
Write $D=\lfloor(d-2\eta)\ell\rfloor$.
Since $0<d\leq 1$ and $0<\eta\le d^2/1000$, we have $D\geq \eta\ell+1.$
Based on \eqref{eq:pair-core-size} and the definition of $V_i(j)$, 
every vertex in $V_i(j)$ has at least
$(p_{ij}-2\eta)\ell\ge D$
neighbors in $V_j(i)$, and the same assertion holds with $i$ and $j$ interchanged.  

 Suppose first that $B_{ij}$ is disconnected.  
 Every connected component of $B_{ij}$ meets both sides of the bipartition $V_i\cup V_j$ and contains at least $D$ vertices in each side.
Hence, two distinct components give subsets
 $X\subseteq V_i(j)$ and $Y\subseteq V_j(i)$ with $\min\{|X|,|Y|\}\geq D\ge\eta\ell$ and $e(X,Y)=0$. 
 Thus, $d(X,Y)=0$ and $|d(X,Y)-d(V_i,V_j)|=d(V_i,V_j)\geq d>\eta$.
 However, the $\eta$-regularity of the pair $(V_i,V_j)$ implies $|d(X,Y)-d(V_i,V_j)|<\eta$,
a contradiction.
Therefore, $B_{ij}$ is connected.

 Suppose next that $B_{ij}$ admits a cut vertex $v$.  
 If $v\in V_i(j)$, then
 every component of $B_{ij}-v$ contains at least $D$ vertices in $V_j(i)$
 and at least $D-1$ vertices in $V_i(j)$; the symmetric statement holds when $v\in V_j(i)$. 
Hence, two distinct components of $B_{ij}-v$ again give subsets $X\subseteq V_i(j)$ and $Y\subseteq V_j(i)$ 
with $\min\{|X|,|Y|\}\geq D-1\ge\eta\ell$ and $e(X,Y)=0$, which 
contradicts the $\eta$-regularity of $(V_i,V_j)$.  
Hence, $B_{ij}$ is connected and has no cut
vertex.
\end{proof}

\begin{lemma}\label{lem:bipartite-skeleton}
Assume $T$ and $G$ satisfy the hypotheses above. Then the graph $B=\bigcup_{ij\in E(T)}B_{ij}$
is a $2$-connected bipartite graph.  
Moreover, we have $n(B)\ge r(1-\eta)\ell.$
\end{lemma}

\begin{proof}
Recall that $T[\ell]$ is the $\ell$-blow-up of the tree $T$.
Clearly, $T[\ell]$ is bipartite, and $B$ is a subgraph of $T[\ell]$.
Consequently, $B$ is bipartite as well.

We now prove that $B$ is 2-connected.
If $e(T)=1$, then $r=2$ and $B=B_{12}$.
By Lemma \ref{lem:pair-core}, $B$ is 2-connected.
If $e(T)\ge 3$, then every edge of $T$ has at least one adjacent edge.
For any two adjacent edges $ij_1,ij_2\in E(T)$, applying \eqref{eq:pair-core-size} yields
 \[
 \big|V_i(j_1)\cap V_i(j_2)\big|\ge \big|V_i(j_1)\big|+\big|V_i(j_2)\big|-\big|V_i\big|\ge\big(1-2\eta\big)\ell\ge2.
 \]
Thus, $B_{ij_1}$ and $B_{ij_2}$ have at least two common vertices, so their union 
$B_{ij_1}\cup B_{ij_2}$ is also 2-connected.
Indeed, the union of any two $2$-connected graphs with at least two common vertices remains $2$-connected: 
deleting any single vertex leaves both graphs connected, and they still share at least one common vertex. 
Induction on $e(T)$ thus confirms that $B$ is $2$-connected.

Since $V(T)=[r]$ and every vertex $i\in [r]$ is incident to at least one edge of $T$, 
it follows that $V(B)\cap V_i$ is non-empty for each $i\in [r]$. Combining \eqref{eq:pair-core-size} yields
$n(B)\geq r(1-\eta)\ell$.
\end{proof}

\subsection{Excluding large bipartite blocks}

A \emph{block} of a graph $G$ is a maximal 2-connected subgraph of $G$.
The following lemma rules out the existence of an almost-spanning bipartite block under the given spectral condition.

\begin{lemma}\label{lem:bipartite-block}
Let $n$ be sufficiently large, and let $G$ be a connected graph with at most $n$ vertices 
such that $\rho(G)>\sqrt{\lfloor{n^2/4}\rfloor}.$
Then any bolck $B$ of $G$ with $n(B)\geq \frac78n$ is non-bipartite.
\end{lemma}

\begin{proof}
 Suppose, for the sake of contradiction, 
 that there exists a bipartite block $B$ with $n(B)\geq \frac78n$. Set $s=n(G)-n(B).$
 If $s=0$, then $G=B$ is bipartite.
 It follows that
$\rho(G)\le\sqrt{e(G)}\le\sqrt{\lfloor{n^2/4}\rfloor},$
a contradiction.  
Hence, $1\le s\le\frac18n$ and $G$ itself is not 2-connected.

Let $G_1,\ldots,G_q$ be the connected components of $G-V(B)$, where $q\geq1$. 
Evidently, each $G_i$ is connected to $B$ by at least one edge.
Furthermore, for every $i\in [q]$, vertices in $G_i$ have exactly one neighbor in $B$.   
Otherwise, a path passing through $G_i$
with two distinct endpoints in $B$ would be a $B$-ear, and the union of this ear
with $B$ would form a strictly larger $2$-connected subgraph, which contradicts the
 maximality of the block $B$.

Set $\rho:=\rho(G)$.  
If vertices in $G_i$ are adjacent to $v_i\in V(B)$, 
let $b_i$ be the $0$--$1$ column vector recording the neighbors of $v_i$ in $G_i$,
and put $n_i=n(G_i)$.  
Since $G_i$ is a proper induced subgraph of the
connected graph $G$, we have $\rho>\rho(G_i)$. Eliminating the Perron
 coordinates on the components $G_i$ gives
 \begin{equation}\label{eq:schur-effective}
  \rho=\rho\bigl(A(B)+\operatorname{diag}(\alpha_v:~v\in V(B))\bigr),
 \end{equation}
 where $\alpha_v$ is the sum of the contributions of the components attached
 at $v$, and the contribution of $G_i$ at $v_i$ is
$\alpha_i=b_i^{\mathsf T}(\rho I-A(G_i))^{-1}b_i.$
 Since $\rho I-A(G_i)$ is positive definite, its inverse has operator
 norm at most $1/(\rho-\rho(G_i))$. Note that
 $\rho(G_i)\le n_i-1$ and $\lVert b_i\rVert^2\le n_i$.
 Thus $\alpha_i\le\frac{n_i}{\rho-n_i+1}
             \le\frac{n_i}{\rho-s+1}.$
 Since $\sum_{i=1}^q n_i=s$, it follows that
 \begin{equation}\label{eq:schur-sum}
  \sum_{i=1}^q\alpha_i\le\frac{s}{\rho-s+1}.
 \end{equation}
 The restriction of a Perron vector of $G$ to $B$ is positive, so the
 eigenvalue in \eqref{eq:schur-effective} is the spectral radius of the
 effective matrix.  By Weyl's inequality, and because a bipartite graph on $n(B)$ vertices
 has spectral radius at most $\frac12n(B)$,
 \begin{equation}\label{eq:block-upper}
 \rho\le\rho(B)+\sum_{i=1}^q\alpha_i
  \le\frac{n-s}{2}+\frac{s}{\rho-s+1}.
 \end{equation}
 On the other hand,
 $\rho>\sqrt{\lfloor{n^2/4}\rfloor}>\frac{1}{2}(n-1).$
 Hence, using $s\le \frac18n$ gives $\rho-s+1>\frac{n+1-2s}{2}>\frac{3}{8}n.$
Since $n$ is sufficiently large, \eqref{eq:block-upper} therefore yields
 \[
 \rho<\frac{n-s}{2}+\frac{8s}{3n}
         \le\frac n2-\frac{5s}{12}
         \le\frac n2-\frac5{12}.
 \]
which contradicts $\sqrt{\lfloor{n^2/4}\rfloor}>\frac12n-\frac14$.
This completes the proof.
\end{proof}

\subsection{A parity-breaking lemma}

Given a graph $G$, let $B$ be a bipartite subgraph of $G$ with a fixed bipartition.  
A \emph{$B$-path} in $G$ is a path with both endpoints in $B$ and no internal vertices from $B$.  
A $B$-path $P$ is said to be \emph{parity-breaking} 
if it is of odd length with both endpoints in the same part of $B$,
or even length with its two endpoints in different parts of $B$.

\begin{lemma}\label{lem:parity-ear}
If $G$ is a $2$-connected non-bipartite graph that contains a $2$-connected bipartite subgraph $B$, 
then $G$ admits a parity-breaking $B$-path $P$ satisfying $\ell(P)\le n(G)-n(B)+1.$
\end{lemma}

\begin{proof}
Since all interval vertices of a parity-breaking $B$-path $P$ lie in $V(G)\setminus V(B)$, 
it follows that $n(P)\leq n(G)-n(B)+2$,
and thus $\ell(P)\le n(G)-n(B)+1.$
Therefore, it suffices to show that there exists at least one parity-breaking $B$-path in $G$.
We proceed by contradiction.
Since $G$ is non-bipartite, there exists an odd cycle $C$ in $G$.
Denote $q=|V(B)\cap V(C)|$.

We first show $q\leq1$. 
Assume for contradiction that $q\ge2$.
The vertices in $V(B)\cap V(C)$ partition $C$ into $q$ edge-disjoint $B$-paths. 
Suppose none of these $B$-paths are parity-breaking: for each such path $P_i$, 
the parity of its length exactly matches the parity of the part membership difference of its two endpoints 
in the bipartition of $B$ (i.e., $\ell(P_i) \equiv 0 \pmod{2}$ if the endpoints lie in the same part of $B$, 
and $\ell(P_i) \equiv 1 \pmod{2}$ otherwise).
Summing the length parity of all $q$ paths around the odd cycle $C$, 
the left-hand side equals $\ell(C) \equiv 1 \pmod{2}$, 
while the right-hand side sums the part membership differences of consecutive vertices 
in $V(B) \cap V(C)$ along $C$, which telescopes to $0 \pmod{2}$.
This leads to a contradiction.

Assume now that $q=1$ and 
$V(B)\cap V(C)=\{x\}$. 
Since $G-x$ is connected, we may choose a shortest path $P$ from $B-x$ to $C-x$ whose internal vertices avoid $B\cup C$. 
Let $y\in V(B)\setminus\{x\}$ and $z\in V(C)\setminus\{x\}$ be its two endpoints.  
The two $x$--$z$ paths in the odd cycle $C$
have opposite length parity. 
Joining each of them to $P$ yields two $x$--$y$ $B$-paths of opposite length parity, 
one of which is parity-breaking.

Finally, suppose that $B\cap C=\varnothing$.  
By the well-known Menger's theorem, 
there exist two vertex-disjoint $B$--$C$ paths with distinct endpoints in $B$ and distinct
endpoints in $C$, and with internal vertices outside $B\cup C$.  
Write them as $xP_1u$ and $yP_2v$, where $x,y\in V(B)$ and $u,v\in V(C)$.  
The two $u$--$v$ paths in the odd cycle $C$ have opposite parity.  
Joining them in turn to $P_1$ and $P_2$
yields two $x$--$y$ $B$-paths of opposite parity, one of which is parity-breaking.
\end{proof}

\subsection{Handling cycles with prescribed parity}

\begin{lemma}\label{lem:one-handle}
Fix parameters $0<d\le1/4$, $0<\eta\le d^2/8000$, and a positive integer $r\ge2$. 
There exists a threshold $\ell_2=\ell_2(d,\eta,r)$ such that the following holds.

Let $T$ be a tree with vertex set $[r]$ that contains a matching $M$ of size $t$. 
Suppose that a graph $G$ contains pairwise disjoint vertex subsets $V_1,\ldots,V_r,$ each of size $\ell\ge \ell_2,$ 
such that for any $ij\in E(T)$, the pair $(V_i,V_j)$ is $\eta$-regular in $G$, with edge density $p_{ij}\geq d$.
Let $P^*$ be a path in $G$ with vertex set $V(P^*)\subseteq \cup_{i\in[r]}V_i$, 
whose distinct endpoints are $x\in V_a$ and $y\in V_b$, such that
 
(i)  $|V(P^*\!-\!\{x,y\})\cap \,V_i|\le\eta\ell$ for every $i\in [r]$;

(ii) there are edges $aa',bb'\in E(T)$ such that $\min\{d_{V_{a'}}(x),d_{V_{b'}}(y)\}\ge(d-\eta)\ell;$

(iii) no edge of $M$ is incident to $a$ or $b$;

(iv) $\ell(P^*)+\operatorname{d}_T(a,b)$ is odd, 
where $d_T(a,b)$ denotes the distance between $a$ and $b$ in $T$.
 
Define
 \begin{equation}\label{eq:SP-def}
  \ell^*:=\Big\lceil{\big(1-2\eta\big)\ell}\Big\rceil
  ~~\text{and} ~~
  S^*:=\Big\lfloor{\big(1-\frac d8-\frac{40\eta}{d}\big)\ell^*}\Big\rfloor.
 \end{equation}
Then there is a positive integer $m\le2r$ such that the
 union of $P^*$ and the edges within the specified $\eta$-regular pairs contains a cycle $C_g$ for every odd length $g$ in
 \begin{equation}\label{eq:one-handle-interval}
  \big[\,\ell(P^*)+m+2t,\ \ell(P^*)+m+2tS^*\,\big].
 \end{equation}
\end{lemma}

\begin{proof}
Choose $\ell_2$ sufficiently large in terms of $d$, $\eta$, and $r$. 
For every $i\in [r]$ and every cluster $V_i$, 
we shall first extract a residual cluster $V^*_i\subseteq (V_i\setminus V(P^*\!-\!\{x,y\}))$, 
ensuring that the regularity persists from every regular pair $(V_i,V_j)$ to its residual pair $(V^*_i,V^*_j)$. 
Condition (i) yields 
$|V_i\setminus V(P^*\!-\!\{x,y\})|\ge(1-\eta)\ell>\ell^*$ for every $i\in [r]$. 
Recall from (\ref{eq:SP-def}) that $\ell^*=\lceil{(1-2\eta)\ell}\rceil$.
 Then for each $i$, we can successfully select a set $V^*_i\subseteq (V_i\setminus V(P^*\!-\!\{x,y\}))$ of size $\ell^*$.

Given an arbitrary edge $ij\in E(T)$.
Recall that the pair $(V_i,V_j)$ is an $\eta$-regular pair with edge density $d(V_i,V_j)\ge d>\eta.$
Set $p_{ij}^*:=d(V^*_i,V^*_j)$.
By Lemma \ref{lem:slicing}, 
the pair $(V^*_i,V^*_j)$ is a $2\eta$-regular pair with edge density $p_{ij}^*\ge d-\eta\ge \frac12d$.  
Note that $|V_i\setminus V^*_i|=\ell-\ell^*\le 2\eta\ell$ for each $i\in [r]$.
Combining condition (ii) yields
 \begin{equation}\label{eq:port-neighbors}
  \min\big\{d_{V^*_{a'}}(x),\,d_{V^*_{b'}}(y)\big\}\ge 
  \min\big\{d_{V_{a'}}(x),\,d_{V_{b'}}(y)\big\}-2\eta\ell\ge (d-3\eta)\ell>2\eta\ell^*+2r.
 \end{equation}

Let $T_M$ be the minimal subtree of $T$ that contains the vertices $a'$, $b'$, 
and all endpoints of the $t$ edges in the matching $M$.  
Form a loopless multigraph $H^*$ on $V(T_M)$ 
by taking one copy of every edge on the $a'$--$b'$ path in $T_M$ and two copies of every other edge of $T_M$. 
If $a'\ne b'$, exactly $a'$ and $b'$ have odd degree in $H^*$; if $a'=b'$, every degree is even. 
Thus, $H^*$ has an Euler trail from $a'$ to $b'$ (an Euler tour when $a'=b'$). 
Adjoin the edge $aa'$ at the beginning and the edge $b'b$ at the end. This gives an $a$--$b$ walk $W$ of length
 \begin{equation}\label{eq:mj-bound}
  m\le e(H^*)+2\le2e(T_M)+2=2n(T_M)\leq 2n(T)=2r.
 \end{equation}
 Every edge of $M$ occurs in $W$ at least once.  Mark one occurrence of
 each edge in $M$.  
 Since the marked underlying edges form a matching, 
 any two marked occurrences are not consecutive in $W$.  
 Condition (iii) ensures that the first and last edge occurrences of $W$ are unmarked. 
 Deleting all marked occurrences splits $W$ into $t+1$ nonempty walk segments
 (if $a=b$, then $W$ is a closed walk that decomposes into $t$ nonempty walk segments).

Regardless of whether $a=b$, we can embed these segments as $t+1$ pairwise vertex-disjoint connector paths 
that link the residual clusters $\{V^*_i\}_{i\in [r]}$, 
with the first vertex occurrence $a$ of $W$ anchored at the fixed vertex $x\in V^*_a$ 
 and the last vertex occurrence $b$ of $W$ anchored at the fixed vertex $y\in V^*_b$.
For every other occurrence $a_h$ of $W$, 
we shall first define a candidate set $X_{h}\subseteq V^*_{a_h}$ as follows. 
If the occurrence $a_h$ is an endpoint of some connector segment, 
it must admits a deleted marked edge $a_ha_h'$. We then define
$$X_{h}:=\Big\{v\in V^*_{a_h}:~ d_{V^*_{a_h'}}(v)\geq(p^*_{a_ha_h'}-2\eta)\ell^*\Big\},$$
i.e., every vertex in $X_{h}$ is typical into $V^*_{a_h'}.$
If no marked edge is adjacent to this vertex occurrence, then we set $X_{h}=V^*_{a_h}$. 
Because marked edge occurrences are not
consecutive, every vertex occurrence is subject to at most one external endpoint condition. 
Recall that $|V^*_i|=\ell^*$ for each $i\in [r]$, 
and $(V^*_i,V^*_j)$ is a $2\eta$-regular pair with edge density $p_{ij}^*$ for any $ij\in E(T)$.  
Lemma \ref{lem:typical} implies 
 \begin{equation}\label{eq:ah-U}
  |X_{h}|\geq(1-2\eta)\ell^*.
 \end{equation}

Having chosen a candidate set $X_{h}\subseteq V^*_{a_h}$ for every internal occurrence $a_h$ of $W$,
we then fix an anchor for $a_h$ from $X_{h}$, 
avoiding all vertices used as anchors in prior steps.
In view of (\ref{eq:mj-bound}), 
the total number of vertex occurrences in $W$ is at most $2r+1.$
Thus, at most $2r$ vertices are forbidden in any residual cluster $V^*_i$ of $G$ at each step of the embedding. 

For a connector segment $a_0a_1\ldots a_q$ with no anchored end,
choose an anchor $v_0\in X_0$ for $a_0$ so that it is typical into $X_1$.   
This is possible because $\min\{|X_{0}|,|X_{1}|\}\geq(1-2\eta)\ell^*>2\eta\ell^*+2r$ and, 
by Lemma \ref{lem:typical}, 
fewer than $2\eta\ell^*$ vertices of $V^*_{a_0}$ are atypical with respect
to $X_1.$

Recall that $p^*_{ij}\ge d-\eta$ for each edge $ij\in E(T)$.
Since $v_0\in X_{0}$ is typical into $X_1$, 
it follows from (\ref{eq:ah-U}) that 
$$d_{X_1}(v_0)\geq (p^*_{a_0a_1}-2\eta)|X_1|\geq(d-3\eta)(1-2\eta)\ell^*>2\eta\ell^*+2r.$$
Furthermore, by Lemma \ref{lem:typical},
fewer than $2\eta\ell^*$ vertices of $V^*_{a_1}$ are atypical with respect
to $X_2.$
Therefore,
we can choose an unused anchor $v_1\in N_G(v_0)\cap X_{1}$ for $a_1$ that is typical into $X_{2}.$ 

Suppose that for all $h\leq q-2$, an unused anchor $v_h\in N_G(v_{h-1})\cap X_{h}$ for $a_h$ has been chosen to be typical into $X_{h+1}$.  
Then we have
$$d_{X_{h+1}}(v_h)\geq(p^*_{a_ha_{h+1}}-2\eta)|X_{h+1}|\geq(d-3\eta)(1-2\eta)\ell^*>2\eta\ell^*+2r.$$
By Lemma \ref{lem:typical}, fewer than $2\eta\ell^*$ vertices of $V^*_{a_{h+1}}$ are atypical with respect to
 $X_{h+2}$. 
Therefore,
we can choose an unused anchor $v_{h+1}\in N_G(v_h)\cap X_{h+1}$ for $a_{h+1}$ that is typical into $X_{h+2}.$ 

In the final step of the current segment, 
the neighborhood $N_G(v_{q-1})\cap X_q$ contains at least $(d-3\eta)(1-2\eta)\ell^*$ vertices, so it
similarly include an unused anchor $v_q$ for $a_q$. 
Thus, we successfully embed this segment into $G$ as the path $v_0v_1\ldots v_q$.

We now embed the first and the last connector segments of $W$.
The first segment starts with an occurrence of $a$, which has been anchored at the fixed vertex $x\in V^*_a$. 
 The subsequent occurrence in this segment is $a'$, and its candidate set is $X_{a'}\subseteq V^*_{a'}$.  
In view of \eqref{eq:port-neighbors}, the neighborhood
$N_G(x)\cap V^*_{a'}$ has more than $2\eta\ell^*+2r$ vertices.  
By Lemma \ref{lem:typical}, 
fewer than $2\eta\ell^*$ vertices of $V^*_{a'}$ are atypical with respect to the next candidate set 
when the segment has another edge occurrence.
Therefore,
we can choose an unused anchor $v_{a'}\in N_G(x)\cap X_{a'}$ for $a'$ that is typical into the next candidate set.
The remainder of this segment is then embedded by the preceding induction step.  
The last segment is embedded in the reverse direction,
starting with the last occurrence of $b$,
via the same argument as for the first segment.

Next we replace $t$ marked edges of $W$ with pairwise vertex-disjoint paths that link the residual clusters $\{V^*_i\}_{i\in [r]}$.
Assume that $pq$ is a marked edge, and its two connector endpoints are embedded into $G$, 
with one anchored at $u\in V^*_p$ and the other at $v\in V^*_q$.
Then by the definition of the endpoint candidate sets, we have
$d_{V^*_q}(u)\geq(p^*_{pq}-2\eta)\ell^*$ and $d_{V^*_p}(v)\geq(p^*_{pq}-2\eta)\ell^*$.
Let $W_{pq}$ denote the set of all other fixed anchors chosen from $V^*_p\cup V^*_q$. 
Since $W$ has at most $2r+1$ vertex occurrences,
we have $|W_{pq}\cap V^*_p|\leq 2r$ and $|W_{pq}\cap V^*_q|\leq 2r$.
Recall that $p_{ij}^*\ge \frac12d$ for each $ij\in E(T)$.  
Apply Lemma \ref{lem:flex-path} to $(V^*_p,V^*_q)$ 
by setting the forbidden set to $W{pq}$ and choosing the parameters
$(d^*,\eta^*,k^*,\ell^*)=(\frac12d,2\eta,2r,\lceil{(1-2\eta)\ell}\rceil)$.
Then by (\ref{eq:S-def}) and (\ref{eq:SP-def}), we have 
$$s^*=s^*(d^*,\eta^*,k^*,\ell^*)=\Big\lfloor\big(1-\frac d8-\frac{40\eta}{d}\big)\ell^*\Big\rfloor=S^*.$$
Moreover, the assumption $\eta\le d^2/8000$ ensures $\eta^*\le {d^*}^2/1000$. 
Therefore, for every $s\in [S^*]$, the marked edge occurrence $pq$
can be replaced in $G$ by a $u$--$v$ path of length $2s+1$ whose
internal vertices all lie in $(V^*_p\cup V^*_q)\setminus W_{pq}$.  
Since the marked edges form a matching $M$ of size $t$, 
these $t$ replacement paths lie in disjoint pairs of clusters and
are mutually vertex-disjoint.

For every $j\in [t]$, we assign the $j$-th replacement path a length of $2s_j+1$, where $s_j\in [S^*]$.
Following the order of the Euler trail $W$,
the $t+1$ embedding paths for unmarked connector segments 
and the $t$ replacement paths for marked edges together form a single $x$--$y$ path $P$.
The embedding paths contribute exactly $m-t$ edges, 
while the replacement paths contribute $\sum_{j=1}^{t}(2s_j+1)$ edges.
The total length of $P$ is thus $m+\sum_{j=1}^t2s_j$.

By the definition of the residual clusters $\{V^*_i\}_{i\in [r]}$, 
the two paths $P$ and $P^*$ are internally vertex-disjoint. 
Consequently, their union 
$P\cup P^*$ forms a cycle $C_g$ of length $\ell(P^*)+m+\sum_{j=1}^t 2s_j.$
 Since the tree $T$ is bipartite, the length $m$ of the $a$--$b$ walk $W$  has the same parity as
 the distance $\operatorname{d}_T(a,b)$.  
 Condition (iv) therefore forces the length $g$ to be odd.  
 Finally, every integer in the interval $[t,tS^*]$ can be expressed as a sum of 
$t$ integers from $[S^*]$, which gives the entire interval \eqref{eq:one-handle-interval}.
\end{proof}

\subsection{Proof for the case of consecutive odd-cycles}

\begin{theorem}\label{thm:odd}
 For every $\varepsilon>0$, there exists an integer $n_0=n_0(\varepsilon)$ such that every graph $G$
 of order $n\ge n_0$ satisfying \eqref{eq:threshold} contains a cycle $C_g$ for
 every odd length
 $3\le g\le(C_0-\varepsilon)n.$
\end{theorem}

\begin{proof}
Recall that $C_0=\frac12(3-\sqrt5)$ and $c_0=\frac12C_0$.
We may assume that $\varepsilon<C_0$, since otherwise the assertion is vacuous.  
Choose positive constants $\alpha,\theta$, and a positive integer $m_0$ so that
\begin{equation}\label{eq:odd-top-pre}
(c_0-\alpha)m_0>3,  \quad 2(c_0-\alpha)(1-\theta)-\frac4{m_0}>C_0-\frac\varepsilon2.
 \end{equation}

 Let $\delta_0=\delta(\alpha)$ be the constant supplied by Lemma
 \ref{lem:matching-extraction}.  
 We first choose $0<\delta<\min\{\delta_0,10^{-4}\},$
 then choose $0<d<\min\{\delta/4,\theta/10,1/4\}$, and finally choose
 $\eta>0$ sufficiently small so that all of the following constraints hold:
 \begin{equation}\label{eq:odd-parameters}
 \begin{split}
  &\eta\le \frac{d^2}{8000}, \quad  d+\eta+\sqrt{2\eta}<\delta, \\
  &\min\Big\{\big(1-\frac d4-\frac{10\eta}{d}\big),\,\big(1-\frac d8-\frac{40\eta}{d}\big)\big(1-2\eta\big)\Big\}>1-\theta,\\
  &2(c_0-\alpha)(1-\theta)(1-\eta)-\frac4{m_0}>C_0-\varepsilon,\\
   &\beta:=1-(1-\eta)^2(1-2\delta)<10^{-3}.
 \end{split}
 \end{equation}
We apply Theorem \ref{thm:regularity} to $G$ with the aforementioned parameters $d,\eta,$ and $m_0$.  
The theorem yields a vertex partition $V(G)=V_0\cup V_1\cup\cdots\cup V_r$  
and the corresponding spanning subgraph $G'$, 
where $V_0$ is the exceptional set of size at most $\eta n$, 
and $V_1,\ldots,V_r$ are the nonexceptional regular clusters of equal size $\ell$.
Consequently,
\begin{equation}\label{eq:rL-lower-odd}
  r\ell=n-|V_0|\ge(1-\eta)n.
 \end{equation}

Let $R(G')$ be the reduced graph associated with $G'$,
which is the graph with vertex set $[r]$ and edge set $\{ij: d(V_i,V_j)>0\}$.
Clearly, $H'$ is a subgraph of the $\ell$-blow-up of $R(G')$.  
Using the same arguments in the proof of Theorem \ref{thm:even} for \eqref{eq:rho-H-odd}, we obtain
 \[
   \ell\cdot \rho(R(G'))\ge\rho(H')>
  \big(\frac12-\delta\big)n
   \ge \big(\frac12-\delta\big)r\ell.
 \]
It follows that $\rho(R(G'))>(\frac12-\delta)r.$
By Lemma \ref{lem:matching-extraction}, 
$R(G')$ has a connected component $R_0$, which contains a matching of size 
 \begin{equation}\label{eq:odd-matching-size}
  t\ge(c_0-\alpha)r.
 \end{equation}

 Suppose first that $R_0$ is non-bipartite. Defined $s^*=\lfloor{(1-\frac d4-\frac{10\eta}{d})\ell\rfloor}$ as in
 \eqref{eq:S-def}.
 Conclusion (i) of Theorem \ref{thm:regularity} bounds $r$ above by
 a constant $M_0=M_0(d,\eta,m_0)$, while $\ell\to\infty$ as $n$ grows.  
 Hence, for sufficiently large $n$, we have both
$\ell\ge \ell_1(d,\eta,2r)$ and $s^*\ge r-1.$
All hypotheses of Corollary \ref{cor:nonbipartite-odd} therefore hold, and $G$
contains a cycle $C_g$ for every odd length $g\in [2r,\,2ts^*]$.
On the other hand, since $r$ is bounded independently of $n$, Theorem
 \ref{thm:short-cycles} covers all odd cycles of smaller lengths.  
Recall $s^*=\lfloor{(1-\frac d4-\frac{10\eta}{d})\ell\rfloor}$.
In view of \eqref{eq:odd-parameters}, \eqref{eq:rL-lower-odd}, and \eqref{eq:odd-matching-size},
we have
 \[
  2ts^*\geq2(c_0-\alpha)rs^*\ge2(c_0-\alpha)r((1-\theta)\ell-1)
  >2(c_0-\alpha)(1-\theta)(1-\eta)-\frac4{m_0}>(C_0-\varepsilon)n
 \]
 for sufficiently large $n$.  The theorem follows in this case.

 We may therefore assume that $R_0$ is bipartite.  Put $r_0=v(R_0)$.  Since
 a bipartite graph of order $r_0$ has spectral radius at most $r_0/2$,
 \[
  \left(\frac12-\delta\right)r<\rho(R_0)\le\frac{r_0}{2},
 \]
 and hence
 \begin{equation}\label{eq:r0-large}
  r_0>(1-2\delta)r.
 \end{equation}
 Extend the matching $M$ to a spanning tree $T$ of $R_0$.  Construct the
 graph $B$ from the tree pairs as in Lemma \ref{lem:bipartite-skeleton}.  By
 \eqref{eq:r0-large}, \eqref{eq:rL-lower-odd}, and
 Lemma \ref{lem:bipartite-skeleton},
 \begin{equation}\label{eq:B-large}
  v(B)\ge(1-\eta)r_0\ell
  >(1-\eta)^2(1-2\delta)n=(1-\beta)n.
 \end{equation}

 Let $C_\rho$ be a component of $G$ with $\rho(C_\rho)=\rho(G)$.  Since
 $\rho(C_\rho)\le v(C_\rho)-1$, we have $v(C_\rho)>n/2$.  The component
 containing $B$ also has more than $n/2$ vertices by \eqref{eq:B-large};
 hence these two components coincide.  Let $Q$ be the block of this component
 containing the $2$-connected graph $B$.  Since $\beta<10^{-3}<1/8$, Lemma
\ref{lem:bipartite-block} shows that $Q$ is non-bipartite.

 By Lemma \ref{lem:parity-ear}, there is a parity-breaking $B$-path $P^*$ such that
 \begin{equation}\label{eq:P-short}
  \ell(P^*)\le v(Q)-v(B)+1\le\beta n+1.
 \end{equation}
 Let its ends be $x\in V_a$ and $y\in V_b$.  Since $x,y\in B$, there are
 edges $aa',bb'\in E(T)$ with
 $x\in V_a(a'),y\in V_b(b').$
 Thus,
 \begin{equation}\label{eq:port-degree-odd}
  d_{V_{a'}}(x),d_{V_{b'}}(y)\ge(d-\eta)\ell.
 \end{equation}
 Since $P^*$ is a $B$-path and $|V_i\setminus V(B)|\le\eta\ell$ for every
 $i\in [r]$,
 \begin{equation}\label{eq:P-cluster-loss}
  \big|V(P^*\!-\!\{x,y\})\cap \,V_i\big|\le\eta\ell.
 \end{equation}

 Delete from $M$ every edge incident with $a$ or $b$, and call the remaining
 matching $M'$.  Then
 \begin{equation}\label{eq:tprime}
  t':=|M'|\ge t-2\ge1.
 \end{equation}
 Order the edges of $M'$ arbitrarily.  The parity-breaking property of $P$
 says precisely that
 \[
  \ell(P^*)+\operatorname{d}_T(a,b)\equiv1\pmod2.
 \]
 All hypotheses of Lemma \ref{lem:one-handle} are now satisfied.  For
 $j=1,\ldots,t'$, let
 \[
  I_j=[\,\ell(P^*)+m_j+2j,\ \ell(P^*)+m_j+2jS^*\,]     
 \]
 be the interval supplied by that lemma.

 These intervals have no odd gap.  Indeed, $m_j\le2r$ gives
 $\min I_j\le\ell(P^*)+2r+2j,$
 whereas
 \[
  \max I_{j-1}\ge\ell(P^*)+2(j-1)S^*.
 \]
 For sufficiently large $n$, we have $S^*\ge r+2$, and hence
 \[
  \max I_{j-1}+2\ge\min I_j
 \]
 for every $j\ge2$.  Moreover, by \eqref{eq:P-short},
 \[
  \min I_1\le\beta n+2r+3<n/320
 \]
 for sufficiently large $n$.  Thus, Theorem \ref{thm:short-cycles} joins the initial
 odd lengths to the interval $I_1\cup\cdots\cup I_{t'}$.

 Finally, \eqref{eq:SP-def}, \eqref{eq:odd-parameters},
 \eqref{eq:odd-matching-size}, \eqref{eq:rL-lower-odd}, and $r\ge m_0$ give
 \begin{align*}
  \max I_{t'}
  &\ge2t'S^*\ge2\bigl((c_0-\alpha)r-2\bigr)(1-\theta)\ell\\
  &\ge2(c_0-\alpha)(1-\theta)(1-\eta)n-\frac{4n}{m_0}\\
  &>(C_0-\varepsilon)n.
 \end{align*}
 Hence, every required odd length occurs.
\end{proof}

\begin{proof}[Proof of Theorem \ref{thm:main}]
 The sharpness assertion is Theorem \ref{prop:sharpness}.  The lower bound follows
 by combining Theorems \ref{thm:even} and \ref{thm:odd}.
\end{proof}

\section*{Acknowledgment}
This project began during 23--26 October 2025, when the first author was visiting the second author at Nanjing University of Science and Technology. The first author is grateful to NUST for the hospitality and warm atmosphere during the visit. The visit led to a weaker result stating that, under the same condition, either all even cycles or all odd cycles \(C_{\ell}\) with \(\ell\in[3,(\frac12(3-\sqrt5)-\varepsilon)n]\) belong to \(G\). The proof makes use of the machinery in \cite{LiNing2023} together with the technique in \cite{ALNS26+}. The current proof is based on a similar proof idea, but with \L{}uczak's technique replaced. The first author is grateful to Prof.~\L{}uczak for sharing his ideas on constructing consecutive odd cycles in an email, before the arXiv version of \cite{ALPZ2023} was submitted. This directly led to the joint work \cite{HLLNP2025} and indirectly to the reading of \cite{Luczak1999}.

\section*{Declaration of AI usage}
The AI (Chatgpt 5.6)
assistant helped with proofreading, grammar checking, and language polishing. In particular, the second author and AI independently found gaps for the original proof of Lemma \ref{lem:closed-walk}. The current version of Lemma \ref{lem:closed-walk} and its proof were originally suggested by AI. The suggested version was rechecked and polished by the authors, which results in the current version. The present authors bear full responsibility
for the correctness of the proofs and the rigor of the writing in this paper.

\end{document}